\documentclass[12pt,reqno]{article}

\usepackage[usenames]{color}
\usepackage{amssymb}
\usepackage{graphicx}
\usepackage{amscd}

\usepackage{lscape}
\usepackage{tikz}
\usepackage{tikz-cd}
\usepackage{tkz-fct}
\usepackage{pgfplots}
\usetikzlibrary{matrix}
\usetikzlibrary{fit,shapes}
\usetikzlibrary{positioning, calc}
\usepackage{colortbl}

\usepackage{amsthm}
\newtheorem{theorem}{Theorem}

\newtheorem{proposition}[theorem]{Proposition}
\newtheorem{corollary}[theorem]{Corollary}

\theoremstyle{definition}

\newtheorem{example}[theorem]{Example}

\usepackage[colorlinks=true,
linkcolor=webgreen, filecolor=webbrown,
citecolor=webgreen]{hyperref}

\definecolor{webgreen}{rgb}{0,.5,0}
\definecolor{webbrown}{rgb}{.6,0,0}

\usepackage{color}
\usepackage{float}
\usepackage{graphics,amsmath,amssymb}
\usepackage{amsfonts}
\usepackage{latexsym}
\usepackage{epsf}

\DeclareMathOperator{\Rev}{Rev}

\newcommand{\seqnum}[1]{\href{http://oeis.org/#1}{\underline{#1}}}

\begin{document}

\begin{center}
\vskip 1cm{\LARGE\bf $d$-orthogonal polynomials, Fuss-Catalan matrices and lattice paths} \vskip 1cm \large
Paul Barry\\
School of Science\\
South East Technological University\\
Ireland\\
\href{mailto:pbarry@wit.ie}{\tt pbarry@wit.ie}
\end{center}
\vskip .2 in

\begin{abstract} In this note, we show how to define certain Riordan arrays, that we call the Fuss-Catalan-Riordan arrays, by means of a special family of $d$-orthogonal polynomials. We relate the Fuss-Catalan Riordan arrays to the Fuss Catalan numbers, and to certain lattice paths.  We emphasise the role of the production matrices of the Riordan arrays that we encounter in our study. \end{abstract}

\section{Preliminaries on generalized Catalan sequences and their generating function}
The Catalan numbers $C_n=\frac{1}{n+1}\binom{2n}{n}$ \seqnum{A000108} are the $2$nd element of a family of sequences $\frac{1}{rn+1} \binom{rn+1}{n}$, where the $0$-th sequence is $\binom{1}{n}$, and the $1$st sequence is the sequence $a_n=1$. The generating function $g_r(x)$ of the $r$-th element of this family satisfies the relation
$$g_r(x)=1+x g_r(x)^r.$$
We have \cite{Conc}
$$g_r(x)=\sum_{n=0}^{\infty} \prod_{j=0}^{n-2} (rn-j)\frac{x^n}{n!}.$$
A lattice path interpretation of the numbers $\frac{1}{rn+1} \binom{rn+1}{n}$ is as follows \cite{Ehr, Conc}.
\begin{proposition} [Raney] The number of lattice paths from $(0,0)$ to $(rn,0)$ consisting of $(r-1)n$ up steps $(1,1)$ and $n$ down steps $(1,1-r)$ that never go below the $x$-axis is the Fuss-Catalan number
$$C_n^r=\frac{1}{rn+1}\binom{rn+1}{n}.$$
\end{proposition}
For $r=-4$ to $r=4$, these generating functions expand to give the sequences that begin
\begin{align*}&1, 1, -4, 26, -204, 1771, -16380, 158224, -1577532, 16112057, -167710664,\ldots,\\
&1, 1, -3, 15, -91, 612, -4389, 32890, -254475, 2017356, -16301164,\ldots,\\
&1, 1, -2, 7, -30, 143, -728, 3876, -21318, 120175, -690690,\ldots,\\
&1, 1, -1, 2, -5, 14, -42, 132, -429, 1430, -4862,\ldots,\\
&1, 1, 0, 0, 0, 0, 0, 0, 0, 0, 0,\ldots,\\
&1, 1, 1, 1, 1, 1, 1, 1, 1, 1, 1,\ldots,\\
&1, 1, 2, 5, 14, 42, 132, 429, 1430, 4862, 16796,\ldots,\\
&1, 1, 3, 12, 55, 273, 1428, 7752, 43263, 246675, 1430715,\ldots,\\
&1, 1, 4, 22, 140, 969, 7084, 53820, 420732, 3362260, 27343888\ldots.\end{align*}
The last four sequences are, respectively, \seqnum{A000012}, \seqnum{A000108}, \seqnum{A0001764} and \seqnum{A002293} in the On-Line Encyclopedia of Integer Sequences (OEIS) \cite{SL1, SL2}.

As we will be dealing with Riordan arrays \cite{book, book2} later in this note, we are interested in the reversion of such sequences and of related sequences. Note that by the \emph{reversion} of a composable power series $f(x)=a_1 x+ a_2x^2+ a_3x^3+\cdots$ (with $a_1 \ne 0$) we shall mean the solution of $f(u)=x$ such that $u(0)=0$. This is the compositional inverse, often denoted $\bar{f}$ or $f^{\langle -1 \rangle}$. We shall also use the notation $\Rev\{f\}$.

We have the following result, which is less general than results to be found in  a more general context \cite{Yang}, but which is sufficient for our purposes.
\begin{proposition} We have
$$\Rev\{x g_r(x)\}=\frac{x}{g_{r-1}(x)}.$$
\end{proposition}
\begin{proof} Letting $\Rev\{x g_r(x)\}=\bar{f}$, we wish to show that $\bar{f}=\frac{x}{g_{r-1}(x)}$. This is equivalent to showing
$$\frac{x}{\bar{f}}=g_{r-1}(x)=1+x g_{r-1}(x)^{r-1}=1+x \left(\frac{x}{\bar{f}}\right)^{r-1},$$
or to showing that
$$x {\bar{f}}^{r-2}={\bar{f}}^{r-1}+x^r.$$
We have
\begin{align*}
(xg_r(x))(\bar{f})&=x\quad\quad\text{by definition}\\
&=\bar{f}g_r(\bar{f})\\
&=\bar{f}(1+\bar{f}g_r(\bar{f})^r)\\
&=\bar{f}+{\bar{f}}^2 g_r(\bar{f})^r\\
&=\bar{f}+(f g_r(\bar{f}))^2 g_r(\bar{f})^{r-2}\\
&=\bar{f}+x^2 g_r(\bar{f})^{r-2}\\
&=\bar{f}+\frac{x^2}{{\bar{f}}^{r-2}} {\bar{f}}^{r-2}g_r(\bar{f})^{r-2}\\
&=\bar{f}+\frac{x^2}{{\bar{f}}^{r-2}} x^{r-2}\\
&=\bar{f}+\frac{x^r}{{\bar{f}}^{r-2}}.\end{align*}
Thus we have that
$$x=\bar{f}+\frac{x^r}{{\bar{f}}^{r-2}},$$ and hence that
$$x {\bar{f}}^{r-2}={\bar{f}}^{r-1}+x^r,$$ as required.
\end{proof}

A second reversion that will be useful is the subject of the next proposition.
\begin{proposition} We have the following reversion
$$\Rev\{xg_r(x)^r\}=\frac{x}{(1+x)^r}.$$
\end{proposition}
\begin{proof}
We want to show that
$$x=(xg_r(x)^r)\left(\frac{x}{(1+x)^r}\right)=\frac{x}{(1+x)^r}\left(g_r\left(\frac{x}{(1+x)^r}\right)^r\right).$$ A necessary and sufficient condition for this to be true is that
$$g_r\left(\frac{x}{(1+x)^r}\right)=1+x.$$
Thus we want
$$\frac{x}{(1+x)^r}g_r\left(\frac{x}{(1+x)^r}\right)=\frac{x}{(1+x)^r} \cdot (1+x)=\frac{x}{(1+x)^{r-1}}.$$
This is equivalent to requiring that
$$(xg_r(x))\left(\frac{x}{(1+x)^r}\right)=\frac{x}{(1+x)^{r-1}}.$$ Reverting, this is equivalent to
\begin{align*}\frac{x}{(1+x)^r}&=(x g_r(x))^{\langle -1 \rangle}\left(\frac{x}{(1+x)^{r-1}}\right)\\
&=\frac{x}{g_{r-1}(x)}\left(\frac{x}{(1+x)^{r-1}}\right)\\
&=\frac{x}{(1+x)^{r-1}}\frac{1}{g_{r-1}\left(\frac{x}{(1+x)^{r-1}}\right)},\end{align*} or that
$$g_{r-1}\left(\frac{x}{(1+x)^{r-1}}\right)=1+x.$$
Now $$g_0\left(\frac{x}{(1+x)^0}\right)=g_0(x)=1+x g_0(x)^0=1+x,$$ so the proof is completed by induction.
Alternatively, we have 
\begin{align*}
x&=\frac{g_r(x)-1}{\frac{g_r(x)-1}{x}}\\
&=\frac{g_r(x)-1}{g_r(x)^r}\quad\quad \text{since $g_r(x)=1+x g_r(x)^r$}\\
&=\frac{g_r(x)-1}{(1+g_r(x)-1)^r}\\
&= \left(\frac{x}{(1+x)^r}\right)(g_r(x)-1).\end{align*}
\end{proof}
Our use of these results is typified by the following result concerning special Riordan arrays.
\begin{proposition}\label{A_seq1} The $A$-sequence of the Bell matrix
$$\left(g_r(x), xg_r(x)\right)$$ is given by
$$A(x)=g_{r-1}(x).$$
\end{proposition}
\begin{proposition} \label{A_seq2} The $A$-sequence of the Riordan array
$$\left(g_r(x), xg_r(x)^r\right)$$ is given by
$$A(x)=(1+x)^r,$$ and the $Z$-sequence is given by
$$Z(x)=(1+x)^{r-1}.$$
\end{proposition}
In the next section, we shall give a short introduction to Riordan arrays. We will give the proof of the last two assertions in that section.

\section{Preliminaries on Riordan arrays}
In this section, we give a brief overview of the elements of the theory of Riordan arrays \cite{book, book2, SGWW}.  To define Riordan arrays, we first of all set
$$\mathcal{F}_r= \{ f \in \mathcal{R}[[x]]\, |\, f(x)=\sum_{k=r}^{\infty} f_k x^k\}.$$
Here, $x$ is a ``dummy variable'' or indeterminate, and $\mathcal{R}$ is any ring in which the operations we will define make sense. Often, it can be any of the fields $\mathbb{Q}, \mathbb{R}$ or $\mathbb{C}$. When looking at combinatorial applications, it can be the ring of integers $\mathbb{Z}$ (in this case we demand that the diagonals of matrices consist of all $1$s).
We now let $g(x) \in \mathcal{F}_0$ and $f(x) \in \mathcal{F}_1$. This means that if $g(x)=\sum_{n=0}^{\infty} g_n x^n$, then $g_0 \ne 0$. Thus $g(x)$ has a multiplicative inverse $\frac{1}{g(x)} \in \mathcal{F}_0$. If $f(x)=\sum_{n=0}^{\infty}$, then we have $f_0=0$ and $f_1 \ne 0$. Then $f(x)$ will have a compositional inverse $\bar{f} \in \mathcal{F}_1$, also denoted by $f^{\langle -1 \rangle}$, which is the solution $v(x)$ of the equation $f(v)=x$ with $v(0)=0$. By definition, we have $\bar{f}(f(x))=x$ and $f(\bar{f}(x))=x$. A \emph{Riordan array} is then defined to be an element $(g, f) \in \mathcal{F}_0 \times \mathcal{F}_1$. The term ``array'' signifies the fact that every element $(g,  f)\in \mathcal{F}_0 \times \mathcal{F}_1$ has a matrix representation $(t_{n,k})_{0 \le n,k \le \infty}$ given by
$$ t_{n,k}=[x^n] g(x) f(x)^k.$$
Here, $[x^n]$ is the functional on $\mathcal{R}[[x]]$ that returns the coefficient of $x^n$ of an element in $\mathcal{R}[[x]]$ \cite{Method}.
With this notation, we can turn the set $\mathcal{F}_0 \times \mathcal{F}_1$ into a group using the product operation
$$(g(x), f(x)) \cdot (u(x), v(x)) = (g(x)u(f(x)), v(f(x)).$$
The inverse for this operation is given by
$$ (g(x), f(x))^{-1}= \left(\frac{1}{g(\bar{f}(x))}, \bar{f}(x)\right).$$
The identity element is given by $(1,x)$, which corresponds to the identity matrix.

When passing to the matrix representation, these operations correspond to matrix multiplication and taking the matrix inverse, respectively. Note that in the matrix representation, the generating functions of the columns are given by the geometric progression $g(x)f(x)^k$ in $\mathcal{R}[[x]]$. The operation of a Riordan array on a power series $h(x)=\sum_{n=0}^{\infty} h_n x^n$ is given by
$$ (g(x), f(x)) \cdot h(x) = g(x)h(f(x)).$$
This weighted composition rule is called \emph{the fundamental theorem of Riordan arrays}.
In matrix terms, this is equivalent to multiplying the column vector $(h_0, h_1, h_2,\ldots)^T$ by the matrix $(t_{n,k})$.
Because $f \in \mathcal{F}_0$, Riordan arrays have lower-triangular matrix representatives. As an abuse of language, we use the term ``Riordan array'' interchangeably to denote either the pair $(g, f)$ or the matrix $(t_{n,k})$, letting the context indicate which is being referred to at the time.
The bivariate generating function of the array $(g, f)$ is given by
$$\frac{g(x)}{1-y f(x)}.$$
Thus we have
$$[x^n] g(x)f(x)^k = [x^n y^k] \frac{g(x)}{1-y f(x)}.$$
To see this, we have
\begin{align*}
[x^n y^k]\frac{g(x)}{1-y f(x)} &=
[x^n y^k]  g(x) \sum_{i=0}^{\infty} y^i f(x)^i\\
&=[x^n] g(x) [y^k]\sum_{i=0}^{\infty} y^i f(x)^i \\
&=[x^n] g(x) f(x)^k.\end{align*}
Setting $y=0$ and $y=1$, respectively, in the generating function, yields the generating functions of the row sums and the diagonal sums of the matrix $(t_{n,k})$. That is, the row sums of the Riordan array $(g(x), f(x))$ have generating function $\frac{g(x)}{1-f(x)}$, while the diagonal sums have generating function $\frac{g(x)}{1-xf(x)}$.

By the \emph{rectification} of a Riordan array we shall mean the (square) matrix whose generating function is given by
$$\frac{g(x)}{1-y\frac{f(x)}{x}}.$$
We shall sometimes denote this by $\left(g(x), \frac{f(x)}{x}\right)$. If $(g(x), f(x))=(g(x),xg(x))$ is a Bell matrix, then this appears as $(g(x), g(x))$. If we have
$t_{n,k}=[x^n] g(x)f(x)^k$ as the $(n,k)$ element of $(g(x), f(x))$, then we have $t_{n+k,k}=[x^n]g(x)\left(\frac{f(x)}{x}\right)^k$ as the $(n,k)$ element of the rectification of $(g(x), f(x))$.
\begin{example} The rectification of the binomial matrix $\mathbf{B}=\left(\binom{n}{k}\right)$ is the matrix $\left(\binom{n+k}{k}\right)$.
\end{example}
If $b(x,y)$ is the generating function of a matrix $M$, where we consider $x$ to be the ``vertical'' variable, and $y$ to be the ``horizontal'' variable, then we call the matrix with generating function $b(x,xy)$ the \emph{downshift} of $M$. 
\begin{example} The binomial matrix $\mathbf{B}$ with generating function $\frac{1}{1-x-xy}$ is the downshift of the matrix $(\binom{n+k}{k})$, which has $\frac{1}{1-x-y}$ as generating function.
\end{example}

An \emph{almost Riordan array} of the first order is a lower-triangular matrix $A=\left(a_{i,j}\right)_{0 \le i,j \le \infty}$ such that the first column $a_{n,0}$ has a generating function $a(x)$ and the embedded matrix $\left(a_{i,j}\right)_{1 \le i,j \le \infty}$ is a Riordan matrix $(g(x), f(x))$. We write $(a(x), g(x), f(x))$ or $(a(x); g(x), f(x))$ for such an almost Riordan array. The set of almost Riordan arrays of first order form a group which properly contains the Riordan group. Similarly, an almost Riordan array of the second order is a lower-triangular matrix which has an embedded Riordan array after the first two columns. Again, the family of such matrices constitutes a group.

An important feature of Riordan arrays \cite{book, book2, Survey, SGWW} is that they have a sequence characterization. Specifically, a lower-triangular array $(t_{n,k})_{0 \le n,k \le \infty}$ is a Riordan array if and only if there exists a sequence $a_n, n\ge 0$, such that
$$t_{n,k}=\sum_{i=0}^{\infty}a_i t_{n-1,k-1+i}.$$
This sum is actually a finite sum, since the matrix is lower-triangular. For such a matrix $M$, we let  $$P=M^{-1}\overline{M},$$ where $\overline{M}$ is the matrix $M$ with its top row removed. Then the matrix $P$ has the form that begins
$$\left(
\begin{array}{ccccccc}
 z_0 & a_0 & 0 & 0 & 0 & 0 & 0 \\
 z_1 & a_1 & a_0 & 0 & 0 & 0 & 0 \\
 z_2 & a_2 & a_1 & a_0 & 0 & 0 & 0 \\
 z_3 & a_3 & a_2 & a_1 & a_0 & 0 & 0 \\
 z_4 & a_4 & a_3 & a_2 & a_1 & a_0 & 0 \\
 z_5 & a_5 & a_4 & a_3 & a_2 & a_1 & a_0 \\
 z_6 & a_6 & a_5 & a_4 & a_3 & a_2 & a_1 \\
\end{array}
\right).$$

Here, $z_0,z_1,z_2,\ldots$ is an ancillary sequence which exists for any lower-triangular matrix. The matrix $P$ is called the \emph{production matrix} of the Riordan matrix $M$. The sequence characterization of renewal arrays was first described by Rogers \cite{He_A, Rogers}.

If $(g(x), f(x))$ is a Riordan array whose $A$ and $Z$-sequences have generating functions $A(x)$ and $Z(x)$, respectively, then we have
$$(g(x), f(x))^{-1}=\left(\frac{A(x)-xZ(x)}{A(x)}, \frac{x}{A(x)}\right),$$ which shows that we can recover a knowledge of $g(x)$ and $f(x)$ from a knowledge of $A(x)$ and $Z(x)$ (at least theoretically). This is so because we have
$$A(x)=\frac{x}{\bar{f}(x)},\quad\quad Z(x)=\frac{1}{\bar{f}(x)}\left(g_0-\frac{1}{g(\bar{f}(x))}\right).$$
We can now give the proofs of Proposition \ref{A_seq1} and Proposition \ref{A_seq2}.
\begin{proof} The generating function of the  $A$-sequence of the Riordan array $(g_r(x), xg_r(x))$ is given by
$$A(x)=\frac{x}{\Rev\{g_r(x)\}}=\frac{x}{\frac{x}{g_{r-1}(x)}}=g_{r-1}(x).$$
For the Riordan array $(g_r(x), xg_r(x)^r)$, we have
$$A(x)=\frac{x}{\Rev\{xg_r(x)^r\}}=\frac{x}{\frac{x}{(1+x)^r}}=(1+x)^r.$$
\end{proof}
In order to prove that $Z=(1+x)^{r-1}$ for the Riordan array $(g_r(x), xg_r(x)^r)$, we proceed as follows.
\begin{proof} We have
$$x=(xg_r(x)^r)\left(\frac{x}{(1+x)^r}\right)=\frac{x}{(1+x)^r}g_r\left(\frac{x}{(1+x)^r}\right)^r.$$
Thus we have $$g_r\left(\frac{x}{(1+x)^r}\right)=1+\frac{x}{(1+x)^r}g_r\left(\frac{x}{(1+x)^r}\right)^r=1+x.$$
Now for the Riordan array $(g_r(x), xg_r(x))$, we have
$$Z(x)=\frac{1}{\frac{x}{(1+x)^r}}\left(1-\frac{1}{g_r\left(\frac{x}{(1+x)^r}\right)}\right),$$ thus
$$Z(x)=\frac{(1+x)^r}{x}\left(1-\frac{1}{1+x}\right)=\frac{(1+x)^r}{x} \frac{x}{1+x}=(1+x)^{r-1}.$$
\end{proof}
In the case of two Riordan arrays $(g_1,f_1)$ and $(g_2,f_2)$ with respective $A$ and $Z$-sequences $A_1, Z_1$ and $A_2, Z_2$, the product array $(g_1,f_1)\cdot (g_2,f_2)$ will have its $A$-sequence given by
$$A_3(x)=\left(A_2(x), \frac{x}{A_2(x)}\right)\cdot A_1(x),$$ and
$$Z_3(x)=\frac{Z_2(x)A_3(x)}{A_2(x)}+Z_1\left(\frac{x}{A_2(x)}\right)\left(1-\frac{xZ_2(x)}{A_2(x)}\right).$$

An alternative approach to the sequence characterization of a Riordan array is to use a matrix characterization \cite{He_M, Merlini}.  One form that such a matrix characterization may take is the following \cite{He_M, Merlini}.
\begin{theorem} A lower-triangular array $(t_{n,k})_{0 \le n,k \le \infty}$ is a Riordan array if and only if there exists another array $A=(a_{i,j})_{i,j \in \mathbb{N}_0}$ with $a_{0,0} \ne 0$, and a sequence $(\rho_j)_{j \in \mathbb{N}_0}$ such that
$$t_{n+1,k+1}=\sum_{i \ge 0} \sum_{j \ge 0} a_{i,j} t_{n-i, k+j} + \sum_{j \ge 0} \rho_j t_{n+1,k+j+2}.$$
\end{theorem}
As we have seen, the power series definition of a Riordan array is as follows. A Riordan array is defined by a pair of power series, $g(x)$ and $f(x)$, where
$$g(x)=g_0 + g_1 x + g_2 x^2+ \cdots, \quad g_0 \ne 0,$$ and
$$f(x)=f_1 x + f_2 x^2+ f_3 x^3+\cdots, \quad f_0=0 \text{ and } f_1 \ne 0.$$
We then have
$$t_{n,k}=[x^n] g(x)f(x)^k,$$ where $[x^n]$ is the functional that extracts the coefficient of $x^n$.
The relationship between $f(x)$ and the pair $(A, \rho)$ is the following.
$$\frac{f(x)}{x}=\sum_{i \ge 0} x^i R^{(i)}(f(x))+\frac{f(x)^2}{x}\rho(f(x)),$$ where
$R^{(i)}$ is the generating series of the $i$-th row of $A$, and $\rho(x)$ is the generating series of the sequence $\rho_n$.
More generally, we have the following scheme.
$$\left(
\begin{array}{ccccccc}
\vdots & \cdots & \cdots & \cdots & \cdots & \cdots & \cdots \\
x^2 & \cdots & t_{n-3,k-1} & t_{n-3,k} & t_{n-3,k+1} & t_{n-3,k+2} & \cdots \\
x & \cdots & t_{n-2,k-1} & t_{n-2,k} & t_{n-2,k+1} & t_{n-2,k+2} & \cdots \\
1 & \cdots & t_{n-1,k-1} & t_{n-1,k} & t_{n-1,k+1} & t_{n-1,k+2} & \cdots \\
1/x & \cdots & t_{n,k-1} & t_{n,k} & t_{n,k+1} & t_{n,k+2} & \cdots \\
1/x^2 & \cdots &  &  &  & t_{n+1,k+2} & \cdots \\
 &  &  &  &  &  &  \\
\cdots & \cdots & 1 & u & u^2 & u^3 & \cdots \\
\end{array}
\right).$$
Thus if, for instance, we have
$$t_{n,k}=t_{n-1,k-1}+ \alpha t_{n-2,k}+ \beta t_{n-2,k+1}+ \gamma t_{n,k+1},$$ with standard initial conditions, then we have
$$\frac{u}{x}=1+\alpha xu+\beta xu^2+ \gamma \frac{u^2}{x},$$ to give
$$u=f(x)=\frac{1-\alpha x^2-\sqrt{1-4 \gamma x - 2\alpha x^2-4 \beta x^3+\alpha^2 x^4}}{2(\gamma+\beta x^2)}.$$
We can interpret this result as saying that the power series $f(x)$ describing the Riordan array that counts lattice paths from $(0,0)$ to $(n,k)$, with step set $\{\alpha*(2,0), \beta*(2,-1), \gamma*(0,-1)\}$ is given by $f(x)=u(x)$. Here, $\alpha*(2,0)$ signifies that the $(2,0)$ steps can have $\alpha$ colors.

\section{Preliminaries on constant coefficient $d$-orthogonal polynomials}
We have seen that for the Riordan array $(g(x), xg_r(x))$, we have
$$Z(x)=(1-x)^{r+1}, A(x)=(1+x)^r.$$ The resulting production matrices are then banded matrices. For instance, if $r=3$, we have
$$Z(x)=1+2x+x^2, A(x)=1+3x+3x^2+x^3,$$ and the production matrix for $(g_3(x), xg_3(x))$ is begins
$$\left(
\begin{array}{ccccccc}
1 &1 & 0 & 0 & 0 & 0 & 0 \\
2 & 3 & 1 & 0 & 0 & 0 & 0 \\
1 & 3 & 3 & 1 & 0 & 0 & 0 \\
0 & 1 & 3 & 3 & 1 & 0 & 0 \\
0 & 0 & 1 & 3 & 3 & 1 & 0 \\
0 & 0 & 0 & 1 & 3 & 3 & 1 \\
0 & 0 & 0 & 0 & 1 & 3 & 3 \\
\end{array}\right).$$ Such matrices occur in the context of $d$-orthogonality, which we now explore in the context of constant coefficient recurrences and Riordan (or almost Riordan) arrays.

For constants $a$,  $b$, and $\alpha$, the recurrence
$$P_n(x)=(x-a)P_{n-1}(x)-b P_{n-2}(x),$$ with initial conditions $P_0(x)=1$, $P_1(x)=x-\alpha$, defines a polynomial sequence whose coefficient array is given by the Riordan array
$$\left(\frac{1+(a-\alpha)x}{1+ax+bx^2}, \frac{x}{1+ax+bx^2}\right).$$
The inverse matrix of this array will then have a tri-diagonal production matrix that begins
$$\left(
\begin{array}{ccccccc}
\alpha &1 & 0 & 0 & 0 & 0 & 0 \\
b &a & 1 & 0 & 0 & 0 & 0 \\
0 &b & a & 1 & 0 & 0 & 0 \\
0 & 0 & b & a & 1 & 0 & 0 \\
0 & 0 & 0 & b & a & 1 & 0 \\
0 & 0 & 0 & 0 & b & a & 1 \\
0 & 0 & 0 & 0 & 0 & b & a \\
\end{array}\right).$$
For constants $a, b, c$ and $\alpha, \beta, \gamma$, the recurrence
$$P_n(x)=(x-a)P_{n-1}(x)-b P_{n-2}(x)-c P_{n-3}(x),$$ with initial conditions $P_0(x)=1$, $P_1(x)=x-\alpha$, $P_2(x)=x^2-\beta x - \gamma$, defines a polynomial sequence $P_n(x)$ whose coefficient array is (in general) an almost Riordan array of first order, whose inverse will have a $4$-diagonal production array that begins
$$\left(
\begin{array}{ccccccc}
\alpha &1 & 0 & 0 & 0 & 0 & 0 \\
\alpha(\beta-\alpha) & \beta-\alpha & 1 & 0 & 0 & 0 & 0 \\
c & b & a & 1 & 0 & 0 & 0 \\
0 & c & b & a & 1 & 0 & 0 \\
0 & 0 & c & b & a & 1 & 0 \\
0 & 0 & 0 & c & b & a & 1 \\
0 & 0 & 0 & 0 & c & b & a \\
\end{array}\right).$$
We call such a family of polynomials $d$-orthogonal, for $d=2$. Ordinary orthogonal polynomials are $1$-orthogonal. Note that we call the inverse of the coefficient array of a family of $d$-orthogonal polynomials the \emph{moment} matrix, in analogy with the situation for orthogonal polynomials.

We note that when $\beta-\alpha=a$, then we get a production matrix that begins
$$\left(
\begin{array}{ccccccc}
\alpha & 1 & 0 & 0 & 0 & 0 & 0 \\
a\alpha & a & 1 & 0 & 0 & 0 & 0 \\
c & b & a & 1 & 0 & 0 & 0 \\
0 & c & b & a & 1 & 0 & 0 \\
0 & 0 & c & b & a & 1 & 0 \\
0 & 0 & 0 & c & b & a & 1 \\
0 & 0 & 0 & 0 & c & b & a \\
\end{array}\right),$$ and hence the coefficient array of the $2$-orthogonal polynomial sequence $P_n(x)$ is given by the Riordan array
$$\left(\frac{1-(a-\alpha)x-(b-a \alpha)x^2}{1+ax+bx^2+cx^2}, \frac{x}{1+ax+bx^2+cx^3}\right).$$
For constants $a,b,c,d$ and $\alpha,\beta,\gamma,\rho,\sigma,\tau$, the recurrence
$$P_n(x)=(x-a)P_{n-1}(x)-b P_{n-2}(x)-c P_{n-3}(x)-d P_{n-4}(x),$$ with initial conditions $P_0(x)=1$, $P_1(x)=x-\alpha$, $P_2(x)=x^2-\beta x - \gamma$, $P_3(x)=x^3-\rho x^2- \sigma x-\tau$, defines a polynomial sequence $P_n(x)$ whose coefficient array is (in general) an almost Riordan array of second order, whose inverse will have a $5$-diagonal production array that begins
$$\left(
\begin{array}{ccccccc}
\alpha  & 1 & 0 & 0 & 0 & 0 & 0 \\
\alpha(\beta-\alpha) & \beta-\alpha & 1 & 0 & 0 & 0 & 0 \\
-\alpha(\beta(\beta-\rho)+\gamma-\sigma)-\gamma(\beta-\rho)+\tau & \beta(\rho-\beta)-\gamma+\sigma & \rho-\beta & 1 & 0 & 0 & 0 \\
d & c & b & a & 1 & 0 & 0 \\
0 & d & c & b & a & 1 & 0 \\
0 & 0 & 0 & c & b & a & 1 \\
0 & 0 & 0 & d & c & b & a \\
\end{array}\right).$$
In the event that we have $\beta-\alpha=a$, $\rho-\beta=a$ and $\beta(\rho-\beta)-\gamma+\sigma=b$, then the coefficient array of the $3$-orthogonal sequence $P_n(x)$ will be the Riordan array
$$\left(\frac{1-(\alpha-a)x-(a \alpha-\beta+\gamma)x^2+(b\alpha +c-\tau)x^3-\gamma x^4}{1+ax+bx^2+cx^3+dx^4}, \frac{x}{1+ax+bx^2+cx^3+dx^4}\right).$$

Generalizations for $d \ge 4$ follow a similar pattern. In general, a $d$-orthogonal polynomial sequence (whose recurrence has constant coefficients) will have a coefficient array that is an almost Riordan array of order $d-1$, which for certain parameter choices may in fact be a Riordan array.

\section{Special $d$-orthogonal polynomials}
We now look at the case of $d$-orthogonal polynomials, whose coefficient array is a Riordan array, where the production array of the inverse of the coefficient array (that is, the moment array) has
$$Z(x)=(1+x)^{r-1}, A(x)=(1+x)^r.$$ Thus we are considering the family of Riordan arrays $(g_r(x), xg_r(x)^r)$.
\begin{example}  When $r=1$, we obtain as moment array the binomial matrix $\mathbf{B}=\left(\binom{n}{k}\right)$. Thus the coefficient array is the inverse binomial matrix $\mathbf{B}^{-1}=\left((-1)^{n-k}\binom{n}{k}\right)$.
\end{example}
\begin{example} When $r=2$, we have $A(x)=(1+x)^2=1+2x+x^2$, and $Z(x)=1+x$. The coefficient matrix is that given by the Riordan array
$$\left(\frac{1}{1+x}, \frac{x}{1+2x+x^2}\right).$$ The corresponding moment array is given by $(c(x), xc(x)^2)$, where $c(x)=\frac{1-\sqrt{1-4x}}{2x}$ is the generating function of the Catalan numbers $C_n=\frac{1}{n+1}\binom{2n}{n}$. This moment matrix \seqnum{A039599} begins
$$\left(
\begin{array}{ccccccc}
1& 0 & 0 & 0 & 0 & 0 & 0 \\
1 & 1 & 0 & 0 & 0 & 0 & 0 \\
2 & 3 & 1 & 0 & 0 & 0 & 0 \\
5 & 9 & 5 & 1 & 0 & 0 & 0 \\
14 &28 & 20 & 7 & 1 & 0 & 0 \\
42 & 90 & 75 & 35 & 9 & 1 & 0 \\
132 & 297 & 275 & 154 & 54 & 11 & 1 \\
\end{array}\right),$$ and its production matrix begins
$$\left(
\begin{array}{ccccccc}
1 &1 & 0 & 0 & 0 & 0 & 0 \\
1 & 2 & 1 & 0 & 0 & 0 & 0 \\
0 & 1 & 2 & 1 & 0 & 0 & 0 \\
0 & 0 & 1 & 2 & 1 & 0 & 0 \\
0 & 0 & 0 & 1 & 2 & 1 & 0 \\
0 & 0 & 0 & 0 & 1 & 2 & 1 \\
0 & 0 & 0 & 0 & 0 & 1 & 2 \\
\end{array}\right).$$
The corresponding family of orthogonal polynomials satisfy
$$P_n(x)=(x-2)P_{n-1}(x)-P_{n-2}(x),$$ subject to $P_0(1)=1$ and $P_1(x)=x-1$. The polynomial sequence begins
$$1, x - 1, x^2 - 3x + 1, x^3 - 5x^2 + 6x - 1, x^4 - 7x^3 + 15x^2 - 10x + 1,\ldots.$$
\end{example}
\begin{example} When $r=3$, we have $A(x)=(1+x)^3=1+3x+3x^2+x^3$, and $Z(x)=(1+x)^2=1+2x+x^2$. Hence the coefficient array is the Riordan array
$$\left(\frac{1}{1+x}, \frac{x}{(1+x)^3}\right),$$ and the moment matrix is the Riordan array
$$\left(\frac{1}{1+x}, \frac{x}{(1+x)^3}\right)^{-1}=(t(x), x t(x)^3),$$ where $t(x)$ is the generating function of the ternary numbers
$t_n=\frac{1}{2n+1}\binom{3n}{n}$. We have
$$t(x)=\frac{2}{\sqrt{3x}}\sin\left(\frac{1}{3}\sin^{-1}\left(\sqrt{\frac{27x}{4}}\right)\right).$$
This moment matrix (a signed version is given by \seqnum{A109956}) begins
$$\left(
\begin{array}{ccccccc}
1& 0 & 0 & 0 & 0 & 0 & 0 \\
1 & 1 & 0 & 0 & 0 & 0 & 0 \\
3 & 4 & 1 & 0 & 0 & 0 & 0 \\
12 & 18 & 7 & 1 & 0 & 0 & 0 \\
55 & 88 & 42 & 10 & 1 & 0 & 0 \\
273 & 455 & 245 & 75 & 13 & 1 & 0 \\
1428 & 2448 & 1428 & 510 & 117 & 16 & 1 \\
\end{array}\right),$$ and its production matrix begins
$$\left(
\begin{array}{ccccccc}
1 &1 & 0 & 0 & 0 & 0 & 0 \\
2 & 3 & 1 & 0 & 0 & 0 & 0 \\
1 & 3 & 3 & 1 & 0 & 0 & 0 \\
0 & 1 & 3 & 3 & 1 & 0 & 0 \\
0 & 0 & 1 & 3 & 3 & 1 & 0 \\
0 & 0 & 0 & 1 & 3 & 3 & 1 \\
0 & 0 & 0 & 0 & 1 & 3 & 3 \\
\end{array}\right).$$
The corresponding family of $2$-orthogonal polynomials satisfies the recurrence
$$P_n(x)=(x-3)P_{n-1}(x)-3P_{n-2}(x)-P_{n-3}(x),$$ subject to $P_0(1)=1$, $P_1(x)=x-1$ and $P_2(x)=x^2-4x+1$.  The polynomial sequence begins
$$1, x - 1, x^2 - 4x + 1, x^3 - 7x^2 + 10x - 1, x^4 - 10x^3 + 28x^2 - 20x + 1,\ldots.$$
\end{example}
\begin{example} When $r=4$, we have $A(x)=(1+x)^4$ and $Z(x)=(1+x)^3$. We find that the coefficient array of the corresponding family of $3$-orthogonal polynomials is given by the Riordan array
$$\left(\frac{1}{1+x}, \frac{x}{(1+x)^4}\right).$$ The moment matrix is then the matrix
 $$\left(\frac{1}{1+x}, \frac{x}{(1+x)^4}\right)^{-1}=(q(x), xq(x)^4),$$ where $q(x)$ is the generating function of the quaternary numbers
$q_n=\frac{1}{3n+1}\binom{4n}{n}$. This Riordan array (for a signed version, see \seqnum{A109962}) begins
$$\left(
\begin{array}{ccccccc}
1& 0 & 0 & 0 & 0 & 0 & 0 \\
1 & 1 & 0 & 0 & 0 & 0 & 0 \\
4 & 5 & 1 & 0 & 0 & 0 & 0 \\
22 & 30 & 9 & 1 & 0 & 0 & 0 \\
140 & 200 & 72 & 13 & 1 & 0 & 0 \\
969 & 1425 & 570 & 130 & 17 & 1 & 0 \\
7084 & 10626 & 4554 & 1196 & 204 & 21 & 1 \\
\end{array}\right),$$ and its production matrix begins
$$\left(
\begin{array}{ccccccc}
1 &1 & 0 & 0 & 0 & 0 & 0 \\
3 & 4 & 1 & 0 & 0 & 0 & 0 \\
3 & 6 & 4 & 1 & 0 & 0 & 0 \\
1 & 4 & 6 & 4 & 1 & 0 & 0 \\
0 & 1 & 4 & 6 & 4 & 1 & 0 \\
0 & 0 & 1 & 4 & 6 & 4 & 1 \\
0 & 0 & 0 & 1 & 4 & 6 & 4 \\
\end{array}\right).$$
The corresponding family of $3$-orthogonal polynomials begins
$$1, x - 1, x^2 - 5x + 1, x^3 - 9x^2 + 15x - 1, x^4 - 13x^3 + 45x^2 - 35x + 1, x^5 - 17x^4 + 91x^3 - 165x^2 + 70x - 1,\ldots.$$ These polynomials satisfy the recurrence
$$P_n(x)=(x-4)P_{n-1}(x)-6P_{n-2}(x)-4P_{n-3}(x)-P_{n-4}(x).$$
\end{example}

\section{The general pre Fuss-Catalan-Riordan matrix}
We wish to calculate the general term of the matrix
$$(g_r(x), xg_r(x)^r)=\left(\frac{1}{1+x}, \frac{x}{(1+x)^r}\right)^{-1}.$$ For this, we use the Lagrange Inversion formula \cite{LI}. The general $(n,k)$-th term $\tau_{n,k}$ of this array is given by
$$\tau_{n,k}=[x^n] \frac{1}{g(\bar{f})} \bar{f}^k,$$ where $g(x)=\frac{1}{1+x}$, and $f(x)=\frac{x}{(1+x)^r}$. Letting
$$G(x)=\frac{x^k}{g(x)},$$ we have
\begin{align*}\tau_{n,k}&=[x^n] \frac{1}{g(\bar{f})} \bar{f}^k\\
&=[x^n]G(\bar{f})\\
&=\frac{1}{n}[x^{n-1}] G'(x)\left(\frac{x}{f(x)}\right)^n\\
&=\frac{1}{n}[x^{n-1}]((k+1)x^k+kx^{k-1})(1+x)^{rn}\\
&=\frac{k+1}{n}\binom{rn}{n-k-1}+\frac{k}{n}\binom{rn}{n-k}\\
&=\frac{rk+1}{(r-1)n+k+1}\binom{rn}{n-k}.\end{align*}
\begin{proposition} The general $(n,k)$ term of the Riordan array $(g_r(x), xg_r(x)^r)$ is given by
$$\tau_{n,k}=\frac{rk+1}{(r-1)n+k+1}\binom{rn}{n-k}.$$
\end{proposition}
We call the Riordan array $\left(g_r(x), xg_r(x)^r\right)$ the pre-Fuss-Catalan-Riordan array.

\section{The Fuss-Catalan numbers: definition via Riordan arrays and examples}
By the Fuss-Catalan numbers we shall understand the rectification $(g_r(x), g_r(x))$ of the ``Fuss-Catalan-Riordan'' Riordan arrays $(g_r(x), xg_r(x))$. In this section, we explore links between the Fuss-Catalan numbers and the Riordan arrays $(g_r(x), xg_r(x)^r)$, and we provide matrix factorizations that allow us to understand the structure of the Fuss-Catalan numbers from different perspectives.

We begin by considering the case of $r=3$. Considering the moment matrix $\left(\frac{1}{1+x}, \frac{x}{(1+x)^3}\right)^{-1}=(t(x), xt(x)^3)$  in this case, and multiplying it on the right by the transpose of the binomial matrix, we get the matrix that begins
$$\left(
\begin{array}{ccccccc}
1& 0 & 0 & 0 & 0 & 0 & 0 \\
1 & 1 & 0 & 0 & 0 & 0 & 0 \\
3 & 4 & 1 & 0 & 0 & 0 & 0 \\
12 & 18 & 7 & 1 & 0 & 0 & 0 \\
55 & 88 & 42 & 10 & 1 & 0 & 0 \\
273 & 455 & 245 & 75 & 13 & 1 & 0 \\
1428 & 2448 & 1428 & 510 & 117 & 16 & 1 \\
\end{array}\right) \cdot \mathbf{B}^T$$
$$=\left(
\begin{array}{ccccccc}
1 &1 & 1 & 1 & 1 & 1 & 1 \\
1 & 2 & 3 & 4 & 5 & 6 & 7 \\
3 & 7 & 12 & 18 & 25 & 33 & 42 \\
12 & 30 & 55 & 88 & 130 & 182 & 245 \\
55 & 143 & 273 & 455 & 700 & 1020 & 1428 \\
273 & 728 & 1428 & 2448 & 3876 & 5814 & 8379 \\
1428 & 3876 &7752 & 13566 & 21945 & 33649 & 49588 \\
\end{array}\right).$$
This is the Fuss-Catalan triangle for $r=3$. Moreover, if we downshift this square matrix to get the triangle that begins
$$\left(
\begin{array}{ccccccc}
1& 0 & 0 & 0 & 0 & 0 & 0 \\
1 & 1 & 0 & 0 & 0 & 0 & 0 \\
3 & 2 & 1 & 0 & 0 & 0 & 0 \\
12 & 7 & 3 & 1 & 0 & 0 & 0 \\
55 & 30 & 12 & 4 & 1 & 0 & 0 \\
273 & 143 & 55 & 18 & 5 & 1 & 0 \\
1428 & 728 & 273 & 88 & 25 & 6 & 1 \\
\end{array}\right),$$
then this is the Riordan array $(t(x), xt(x))$. The corresponding production matrix has $A(x)=c(x)$, and begins
$$\left(
\begin{array}{ccccccc}
1& 1 & 0 & 0 & 0 & 0 & 0 \\
2 & 1 & 1 & 0 & 0 & 0 & 0 \\
5 & 2 & 1 & 1 & 0 & 0 & 0 \\
14 & 5 & 2 & 1 & 1 & 0 & 0 \\
42 & 14 & 5 & 2 & 1 & 1 & 0 \\
132 & 42 & 14 & 5 & 2 & 1 & 1 \\
429 & 132 & 42 & 14 & 5 & 2 & 1 \\
\end{array}\right).$$
In addition, we have the following product expression for the Riordan array $(t(x),xt(x))$.
$$\left(
\begin{array}{ccccccc}
1& 0 & 0 & 0 & 0 & 0 & 0 \\
1 & 1 & 0 & 0 & 0 & 0 & 0 \\
3 & 2 & 1 & 0 & 0 & 0 & 0 \\
12 & 7 & 3 & 1 & 0 & 0 & 0 \\
55 & 30 & 12 & 4 & 1 & 0 & 0 \\
273 & 143 & 55 & 18 & 5 & 1 & 0 \\
1428 & 728 & 273 & 88 & 25 & 6 & 1 \\
\end{array}\right)=$$
$$\left(
\begin{array}{ccccccc}
1& 0 & 0 & 0 & 0 & 0 & 0 \\
1 & 1 & 0 & 0 & 0 & 0 & 0 \\
3 & 4 & 1 & 0 & 0 & 0 & 0 \\
12 & 18 & 7 & 1 & 0 & 0 & 0 \\
55 &88 & 42 & 10 & 1 & 0 & 0 \\
273 & 455 & 245 & 75 & 13 & 1 & 0 \\
1428 & 2448 & 1428 & 510 & 117 & 16 & 1 \\
\end{array}\right) \cdot \left(1, \frac{x}{(1+x)^2}\right),$$ that is,
$$(t(x), xt(x))= (t(x), xt(x)^3) \cdot \left(1, \frac{x}{(1+x)^2}\right).$$
Equivalently, we have
\begin{proposition}
$$(t(x), xt(x))=\left(\frac{1}{1+x}, \frac{x}{(1+x)^3}\right)^{-1} \cdot  \left(1, \frac{x}{(1+x)^2}\right).$$
\end{proposition}
\begin{proof}
We have
\begin{align*}
(t(x), xt(x)^3)\cdot \left(1,\frac{x}{(1+x)^2}\right)&=\left(t(x), \frac{xt(x)^3}{(1+xt(x)^3)^2}\right)\\
&=\left(t(x), \frac{xt(x)^3}{t(x)^2}\right)\\
&=(t(x), xt(x)).\end{align*}
\end{proof}
We call $(t(x), xt(x))$ the Fuss-Catalan-Riordan matrix for $r=3$, and we consider the product of matrices above the canonical factorization of this matrix.

For the case $r=4$, we consider the product
$$\left(
\begin{array}{ccccccc}
1& 0 & 0 & 0 & 0 & 0 & 0 \\
1 & 1 & 0 & 0 & 0 & 0 & 0 \\
4 & 5 & 1 & 0 & 0 & 0 & 0 \\
22 & 30 & 9 & 1 & 0 & 0 & 0 \\
140 & 200 & 72 & 13 & 1 & 0 & 0 \\
969 & 1425 & 570 & 130 & 17 & 1 & 0 \\
7084 & 10626 & 4554 & 1196 & 204 & 21 & 1 \\
\end{array}\right)\cdot \mathbf{B}^T$$

$$=\left(
\begin{array}{ccccccc}
1 &1 & 1 & 1 & 1 & 1 & 1 \\
1 & 2 & 3 & 4 & 5 & 6 & 7 \\
4 & 9 & 15 & 22 & 30 & 39 & 49 \\
22 & 52 & 91 & 140 & 200 & 272 & 357 \\
140 & 340 & 612 & 969 & 1425 & 1995 & 2695 \\
969 & 2394 & 4389 & 7084 & 10626 & 15180 & 20930 \\
7084 & 17710 &7752 & 32890 & 81900 & 118755 & 166257 \\
\end{array}\right).$$
This is the Fuss-Catalan matrix for $r=4$. If we downshift this square matrix to get the triangle that begins
$$\left(
\begin{array}{ccccccc}
1& 0 & 0 & 0 & 0 & 0 & 0 \\
1 & 1 & 0 & 0 & 0 & 0 & 0 \\
4 & 2 & 1 & 0 & 0 & 0 & 0 \\
22 & 9 & 3 & 1 & 0 & 0 & 0 \\
140 & 52 & 15 & 4 & 1 & 0 & 0 \\
969 & 340 & 91 & 22 & 5 & 1 & 0 \\
7084 & 2394 & 612 & 140 & 30 & 6 & 1 \\
\end{array}\right).$$  This is the Riordan array $(q(x), xq(x))$. The corresponding production matrix has $A(x)=t(x)$, and begins
$$\left(
\begin{array}{ccccccc}
1& 1 & 0 & 0 & 0 & 0 & 0 \\
3 & 1 & 1 & 0 & 0 & 0 & 0 \\
12 & 3 & 1 & 1 & 0 & 0 & 0 \\
55 & 12 & 3 & 1 & 1 & 0 & 0 \\
273 & 55 & 12 & 3 & 1 & 1 & 0 \\
1428 & 273 & 55 & 12 & 3 & 1 & 1 \\
7752 & 1428 & 273 & 55 & 12 & 3 & 1 \\
\end{array}\right).$$
Then we have
$$\left(
\begin{array}{ccccccc}
1& 0 & 0 & 0 & 0 & 0 & 0 \\
1 & 1 & 0 & 0 & 0 & 0 & 0 \\
4 & 2 & 1 & 0 & 0 & 0 & 0 \\
22 & 9 & 3 & 1 & 0 & 0 & 0 \\
140 & 52 & 15 & 4 & 1 & 0 & 0 \\
969 & 340 & 91 & 22 & 5 & 1 & 0 \\
7084 & 2394 & 612 & 140 & 30 & 6 & 1 \\
\end{array}\right)$$
$$=\left(
\begin{array}{ccccccc}
1& 0 & 0 & 0 & 0 & 0 & 0 \\
1 & 1 & 0 & 0 & 0 & 0 & 0 \\
4 & 5 & 1 & 0 & 0 & 0 & 0 \\
22 & 30 & 9 & 1 & 0 & 0 & 0 \\
140 & 200 & 72 & 13 & 1 & 0 & 0 \\
969 & 1425 & 570 & 130 & 17 & 1 & 0 \\
7084 & 10626 & 4554 & 1196 & 204 & 21 & 1 \\
\end{array}\right) \cdot \left(1, \frac{x}{(1+x)^3}\right).$$ That is,
$$(q(x), xq(x))= (q(x), x q(x)^4) \cdot  \left(1, \frac{x}{(1+x)^3}\right).$$
Equivalently, we have
$$(q(x), xq(x))=\left(\frac{1}{1+x}, \frac{x}{(1+x)^4}\right)^{-1} \cdot  \left(1, \frac{x}{(1+x)^3}\right).$$ We state this as a proposition.
Equivalently, we have
\begin{proposition}
$$(q(x), xq(x))=\left(\frac{1}{1+x}, \frac{x}{(1+x)^4}\right)^{-1} \cdot  \left(1, \frac{x}{(1+x)^3}\right).$$
\end{proposition}
\begin{proof}
We have
\begin{align*}
(q(x), xq(x)^4)\cdot \left(1,\frac{x}{(1+x)^3}\right)&=\left(q(x), \frac{xq(x)^4}{(1+xq(x)^4)^3}\right)\\
&=\left(q(x), \frac{xq(x)^4}{q(x)^3}\right)\\
&=(q(x), xq(x)).\end{align*}
\end{proof}
We call $(q(x), xq(x))$ the Fuss-Catalan-Riordan matrix for $r=4$, and we consider the product of matrices above the canonical factorization of this matrix.

In general, we have the following result.
\begin{proposition} Let $g_r(x)=1+xg_r(x)^r$. Then
$$(g_r(x), x g_r(x))=(g_r(x), x g_r(x)^r)\cdot \left(1, \frac{x}{(1+x)^{r-1}}\right).$$
\end{proposition}
\begin{proof}
We have
\begin{align*}
(g_r(x), x g_r(x)^r)\cdot \left(1, \frac{x}{(1+x)^{r-1}}\right)&=\left(g_r(x), \frac{xg_r(x)^r}{(1+xg_r(x)^r)^{r-1}}\right)\\
&=\left(g_r(x), \frac{xg_r(x)^r}{g_r(x)^{r-1}}\right)\\
&=(g_r(x), x g_r(x)).\end{align*}
\end{proof}
We can use this result to say something about the $A$-sequence of the Riordan array $\left(1, \frac{x}{(1+x)^{r-1}}\right)$. If we denote this by $A(x)$, then we have
$$\left(A(x), \frac{x}{A(x)}\right)\cdot (1+x)^r=g_{r-1}(x).$$ Equivalently, we have
$$A(x)\left(1+\frac{x}{A(x)}\right)^r = g_{r-1}(x).$$
In general, we have that the Fuss-Catalan numbers are the result of multiplying the Riordan array $\left(\frac{1}{1+x}, \frac{x}{(1+x)^{r+1}}\right)^{-1}$ by $\mathbf{B}^T$. Thus the general term of the Fuss-Catalan matrix can be expressed as
$$\sum_{j=0}^n \frac{rj+1}{(r-1)n+j+1} \binom{rn}{n-j}\binom{k}{j}.$$
\begin{proposition} The Fuss-Catalan matrix $(g_r(x), g_r(x))$ is the result of multiplying the Riordan array $(g_r(x), x g_r(x)^r)=\left(\frac{1}{1+x}, \frac{x}{(1+x)^r}\right)^{-1}$ on the right by the transpose of the binomial matrix $\mathbf{B}$.
\end{proposition}
\begin{proof} The bivariate generating function of $(g_r(x), x g_r(x)^r)$ is given by
$$\frac{g(x)}{1-yxg_r(x)^r}.$$ Multiplying on the right by $\left(\frac{1}{1-y}, \frac{y}{1-y}\right)^T$ gives us
\begin{align*}\frac{1}{1-y} \frac{g_r(x)}{1-\frac{y}{1-y}xg_r(x)^r}&=\frac{g_r(x)}{1-y-yxg_r(x)^r}\\
&=\frac{g_r(x)}{1-y(1+xg_r(x)^r)}\\
&=\frac{g_r(x)}{1-yg_r(x)}.\end{align*}
This is the generating function of $(g_r(x), g_r(x))$ as required.
\end{proof}

\section{More on the general case}

We have defined the $r$-th Fuss-Catalan-Riordan array to be the Riordan array $(g_r(x), xg_r(x))$, and we have defined the Fuss-Catalan matrix to be the rectification of this Riordan array. The first five Fuss-Catalan-Riordan arrays begin as follows.
$$\left(
\begin{array}{ccccc}
1& 0 & 0 & 0 & 0 \\
1 & 1 & 0 & 0 & 0 \\
0 & 2 & 1 & 0 & 0 \\
0 & 1 & 3 & 1 & 0  \\
0 & 0 & 3 & 4 & 1 \\
\end{array}\right),\left(
\begin{array}{ccccc}
1& 0 & 0 & 0 & 0 \\
1 & 1 & 0 & 0 & 0 \\
1 & 2 & 1 & 0 & 0 \\
1 & 3 & 3 & 1 & 0  \\
1 &4 & 6 & 4 & 1 \\
\end{array}\right),\left(
\begin{array}{ccccc}
1& 0 & 0 & 0 & 0 \\
1 & 1 & 0 & 0 & 0 \\
2 & 2 & 1 & 0 & 0 \\
5 & 5 & 3 & 1 & 0  \\
14 &14 & 9 & 4 & 1 \\
\end{array}\right),\left(
\begin{array}{ccccc}
1& 0 & 0 & 0 & 0 \\
1 & 1 & 0 & 0 & 0 \\
3 & 2 & 1 & 0 & 0 \\
12 & 7 & 3 & 1 & 0  \\
55 &30 & 12 & 4 & 1 \\
\end{array}\right),$$
$$\left(
\begin{array}{ccccc}
1& 0 & 0 & 0 & 0 \\
1 & 1 & 0 & 0 & 0 \\
4 & 2 & 1 & 0 & 0 \\
22 & 9 & 3 & 1 & 0  \\
140 &52 & 15 & 4 & 1 \\
\end{array}\right).$$ Their respective production matrices are then as follows.
$$\left(
\begin{array}{ccccc}
1&  1 & 0 & 0 & 0 \\
-1 & 1 & 1 & 0 & 0 \\
2 & -1 & 1 & 1 & 0 \\
-5 & 2 & -1 & 1 & 1  \\
14 & -5 & 2 & -1 & 1 \\
\end{array}\right),
\left(
\begin{array}{ccccc}
1&  1 & 0 & 0 & 0 \\
0 & 1 & 1 & 0 & 0 \\
0 & 0 & 1 & 1 & 0 \\
0 & 0 & 0 & 1 & 1  \\
0 & 0 & 0 & 0 & 1 \\
\end{array}\right),\left(
\begin{array}{ccccc}
1& 1 & 0 & 0 & 0 \\
1 & 1 & 1 & 0 & 0 \\
1 & 1 & 1 & 1 & 0 \\
1 & 1 & 1 & 1 & 1  \\
1 & 1 & 1 & 1 & 1 \\
\end{array}\right),\left(
\begin{array}{ccccc}
1& 1 & 0 & 0 & 0 \\
2 & 1 & 1 & 0 & 0 \\
5 & 2 & 1 & 1 & 0 \\
14 & 5 & 2 & 1 & 1  \\
42 &14 & 5 & 2 & 1 \\
\end{array}\right),$$
$$\left(
\begin{array}{ccccc}
1& 1 & 0 & 0 & 0 \\
3 & 1 & 1 & 0 & 0 \\
12 & 3 & 1 & 1 & 0 \\
55 & 12 & 3 & 1 & 1  \\
273 &55 & 12 & 3 & 1 \\
\end{array}\right).$$
The general term of the Fuss-Catalan-Riordan array is given by
$$FCR(n,k;r)=\frac{k+1}{r(n-k)+k+1}\binom{r(n-k)+k+1}{n-k}.$$
The rectification of this, that is, the Fuss-Catalan matrix, is then given by
$$FC(n,k;r)=\frac{k+1}{rn+k+1}\binom{rn+k+1}{n}.$$
The first five matrices $(g_r(x), xg_r(x)^r)$ begin
\begin{scriptsize}
$$\left(
\begin{array}{ccccc}
1& 0 & 0 & 0 & 0 \\
1 & 1 & 0 & 0 & 0 \\
0 & 1 & 1 & 0 & 0 \\
0 & 0 & 1 & 1 & 0  \\
0 & 0 & 0 & 1 & 1 \\
\end{array}\right),
\left(
\begin{array}{ccccc}
1& 0 & 0 & 0 & 0 \\
1 & 1 & 0 & 0 & 0 \\
1 & 2 & 1 & 0 & 0 \\
1 & 3 & 3 & 1 & 0  \\
1 & 4 & 6 & 4 & 1 \\
\end{array}\right),
\left(
\begin{array}{ccccc}
1& 0 & 0 & 0 & 0 \\
1 & 1 & 0 & 0 & 0 \\
2 & 3 & 1 & 0 & 0 \\
5 & 9 & 5 & 1 & 0  \\
14 & 28 & 20 & 7 & 1 \\
\end{array}\right),
\left(
\begin{array}{ccccc}
1& 0 & 0 & 0 & 0 \\
1 & 1 & 0 & 0 & 0 \\
3 & 4 & 1 & 0 & 0 \\
12 & 18 & 7 & 1 & 0  \\
55 & 88 & 42 & 10 & 1 \\
\end{array}\right),
\left(
\begin{array}{ccccc}
1& 0 & 0 & 0 & 0 \\
1 & 1 & 0 & 0 & 0 \\
4 & 5 & 1 & 0 & 0 \\
22 & 10 & 9 & 1 & 0  \\
140 & 20 & 72 & 13 & 1 \\
\end{array}\right).$$
\end{scriptsize}
Their respective production matrices are then as follows.
\begin{scriptsize}
$$\left(
\begin{array}{ccccc}
1&   1 & 0 &  0 & 0 \\
-1 & 0 & 1 &  0 & 0 \\
1 &  0 & 0 &  1 & 0 \\
-1 & 0 & 0 &  0 & 1  \\
1 & 0 & 0 &  0 & 0 \\
\end{array}\right),
\left(
\begin{array}{ccccc}
1&   1 & 0 &  0 & 0 \\
0 & 1 & 1 &  0 & 0 \\
0 &  0 & 1 &  1 & 0 \\
0 & 0 & 0 &  1 & 1  \\
0 & 0 & 0 &  0 & 1 \\
\end{array}\right),
\left(
\begin{array}{ccccc}
1&   1 & 0 &  0 & 0 \\
1 & 2 & 1 &  0 & 0 \\
0 &  1 & 2 &  1 & 0 \\
0 & 0 & 1 &  2 & 1  \\
0 & 0 & 0 &  1 & 2 \\
\end{array}\right),\left(
\begin{array}{ccccc}
1&   1 & 0 &  0 & 0 \\
2 & 3 & 1 &  0 & 0 \\
1 &  3 & 3 &  1 & 0 \\
0 & 1 & 3 &  3 & 1  \\
0 & 0 & 1 &  3 & 3 \\
\end{array}\right),
\left(
\begin{array}{ccccc}
1&   1 & 0 &  0 & 0 \\
3 & 4 & 1 &  0 & 0 \\
3 & 6 & 4 &  1 & 0 \\
1 & 4 & 6 &  4 & 1  \\
0 & 1 & 4 &  6 & 4 \\
\end{array}\right).$$
\end{scriptsize}
We see that apart from the first matrix, these are banded matrices, indicating the presence of $d$-orthogonality.

\section{Lattice paths}
We consider the case of the ternary numbers $t_n=\frac{1}{2n+1}\binom{3n}{n}$.
The ternary numbers $t_n$ \seqnum{A001764},  correspond to lattice paths from $(0,0)$ to $(n,0)$ with step set $\{(1,1), (-1,-2)\}$. This can be seen because we have
$$\frac{u}{x}=1+\frac{u^3}{x^2},$$ or, equivalently,
$$g=1+xg^3,$$ where $g(x)=\frac{u}{x}$. The corresponding left factor matrix is the ternary matrix $(t(x), xt(x))$ \seqnum{A110616} that begins
$$\left(
\begin{array}{ccccccc}
 1 & 0 & 0 & 0 & 0 & 0 & 0 \\
 1 & 1 & 0 & 0 & 0 & 0 & 0 \\
 3 & 2 & 1 & 0 & 0 & 0 & 0 \\
12 & 7 & 3 & 1 & 0 & 0 & 0 \\
55 & 30 & 12 & 4 & 1 & 0 & 0 \\
273 & 143 & 55 & 18 & 5 & 1 & 0 \\
1428 & 728 & 273 & 88 & 25 & 6 & 1 \\
\end{array}\right),$$
where
$$t(x)=\frac{2}{\sqrt{3x}} \sin\left(\frac{1}{3}\sin^{-1}\left(\frac{\sqrt{27x}}{2}\right)\right)$$ is the generating function of ternary numbers. The ternary matrix (see \textbf{Figure} \ref{Ter}) satisfies the recurrence
$$t_{n,k}=t_{n-1,k-1}+t_{n+1,k+2},$$ where $t_{n,k}=0$ if $n<0$ or $k<0$ or $k>n$, and $t_{0,0}=1$. The corresponding Fuss-Catalan matrix for $r=3$ corresponds to the recurrence 
$$t_{n,k}=t_{n,k-1}+t_{n-1,k+2}.$$ This counts lattice paths form $(0,0)$ to $(n,k)$ with step set $\{(0,1), (1,-2)\}$. See \textbf{Figure} \ref{TerM}.

The quaternary numbers $q_n=\frac{1}{3n+1}\binom{4n}{n}$ arise from the relation
$$\frac{u}{x}=1+\frac{u^4}{x^3},$$ or, equivalently,
$$g=1+xg^4.$$ These then correspond to lattice paths from $(0,0)$ to $(n,0)$ with step set $\{(1,1), (-2,-3)\}$, and the left-factor matrix for these paths is the matrix $(q(x), xq(x))$.

In general, the Fuss-Catalan-Riordan array is the lattice path matrix for the step set $\{(1,1), (2-r,1-r)\}$. Interpretations in terms of path pairs are given in \cite{Drube}. Similarly, the Fuss-Catalan matrices correspond to lattice paths with step set $\{(0,1), (1,1-r)\}$. 
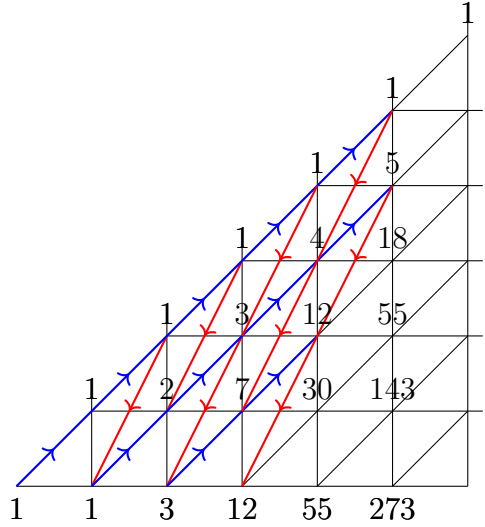
\begin{figure}
\begin{center}
\begin{tikzpicture}
\draw(0,0)--(6,0);
\draw(0,0)--(6,6);
\draw(1,0)--(1,1.2);
\draw(2,0)--(2,2.2);
\draw(3,0)--(3,3.2);
\draw(4,0)--(4,4.2);
\draw(5,0)--(5,5.2);
\draw(6,0)--(6,6.2);

\draw(1,1)--(6.2,1);
\draw(2,2)--(6.2,2);
\draw(3,3)--(6.2,3);
\draw(4,4)--(6.2,4);
\draw(5,5)--(6.2,5);

\draw(1,0)--(6,5);
\draw(2,0)--(6,4);
\draw(3,0)--(6,3);
\draw(4,0)--(6,2);
\draw(5,0)--(6,1);

\draw[red, thick, postaction={decorate},
      decoration={markings, mark=at position 0.5 with {\arrow{>}}}]
      (2,2) -- (1,0);

\draw[red, thick, postaction={decorate},
      decoration={markings, mark=at position 0.5 with {\arrow{>}}}]
      (3,3) -- (2,1);

\draw[red, thick, postaction={decorate},
      decoration={markings, mark=at position 0.5 with {\arrow{>}}}]
      (3,2) -- (2,0);

\draw[blue, thick, postaction={decorate},
      decoration={markings, mark=at position 0.5 with {\arrow{>}}}]
      (0,0) -- (1,1);

\draw[blue, thick, postaction={decorate},
      decoration={markings, mark=at position 0.5 with {\arrow{>}}}]
      (1,1) -- (2,2);

\draw[blue, thick, postaction={decorate},
      decoration={markings, mark=at position 0.5 with {\arrow{>}}}]
      (2,2) -- (3,3);

\draw[blue, thick, postaction={decorate},
      decoration={markings, mark=at position 0.5 with {\arrow{>}}}]
      (2,1) -- (3,2);

\draw[blue, thick, postaction={decorate},
      decoration={markings, mark=at position 0.5 with {\arrow{>}}}]
      (1,0) -- (2,1);

\draw[blue, thick, postaction={decorate},
      decoration={markings, mark=at position 0.5 with {\arrow{>}}}]
      (3,1) -- (4,2);

\draw[blue, thick, postaction={decorate},
      decoration={markings, mark=at position 0.5 with {\arrow{>}}}]
      (2,0) -- (3,1);

\draw[blue, thick, postaction={decorate},
      decoration={markings, mark=at position 0.5 with {\arrow{>}}}]
      (3,2) -- (4,3);

\draw[blue, thick, postaction={decorate},
      decoration={markings, mark=at position 0.5 with {\arrow{>}}}]
      (4,3) -- (5,4);

\draw[blue, thick, postaction={decorate},
      decoration={markings, mark=at position 0.5 with {\arrow{>}}}]
      (3,3) -- (4,4);

\draw[blue, thick, postaction={decorate},
      decoration={markings, mark=at position 0.5 with {\arrow{>}}}]
      (4,4) -- (5,5);

\draw[red, thick, postaction={decorate},
      decoration={markings, mark=at position 0.5 with {\arrow{>}}}]
      (5,4) -- (4,2);

\draw[red, thick, postaction={decorate},
      decoration={markings, mark=at position 0.5 with {\arrow{>}}}]
      (4,2) -- (3,0);

\draw[red, thick, postaction={decorate},
      decoration={markings, mark=at position 0.5 with {\arrow{>}}}]
      (4,4) -- (3,2);

\draw[red, thick, postaction={decorate},
      decoration={markings, mark=at position 0.5 with {\arrow{>}}}]
      (5,5) -- (4,3);

\draw[red, thick, postaction={decorate},
      decoration={markings, mark=at position 0.5 with {\arrow{>}}}]
      (4,3) -- (3,1);

\node at (0,-.3) {1};
\node at (1,1.3) {1};
\node at (2,2.3) {1};
\node at (3,3.3) {1};
\node at (4,4.3) {1};
\node at (5,5.3) {1};
\node at (6,6.3) {1};

\node at (0,-.3) {1};
\node at (1,-.3) {1};
\node at (2,-.3) {3};
\node at (3,-.3) {12};
\node at (4,-.3) {55};
\node at (5,-.3) {273};

\node at (1,1.3) {1};
\node at (2,2.3) {1};
\node at (3,3.3) {1};
\node at (4,4.3) {1};
\node at (5,5.3) {1};
\node at (6,6.3) {1};

\node at (1,-.3) {1};
\node at (2,-.3) {3};
\node at (3,-.3) {12};
\node at (4,-.3) {55};
\node at (5,-.3) {273};

\node at (2,1.3) {2};

\node at (3,1.3) {7};
\node at (3,2.3) {3};

\node at (4,1.3) {30};
\node at (4,2.3) {12};
\node at (4,3.3) {4};

\node at (5,1.3) {143};
\node at (5,2.3) {55};
\node at (5,3.3) {18};
\node at (5,4.3) {5};
\end{tikzpicture}
\end{center}
\caption{The Fuss-Catalan-Riordan array (for $r=3$) $(t(x), xt(x))$ as a path matrix}\label{Ter}
\end{figure}

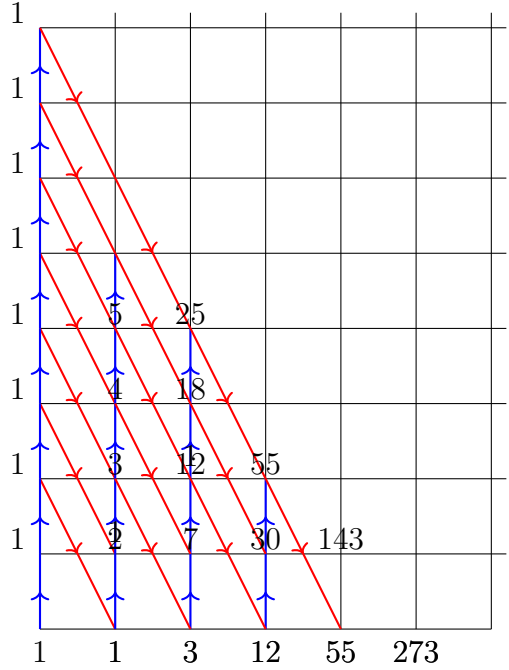
\begin{figure}
\begin{center}
\begin{tikzpicture}
\draw(0,0)--(0,8);
\draw(0,0)--(6,0);
\draw(1,0)--(1,8.2);
\draw(2,0)--(2,8.2);
\draw(3,0)--(3,8.2);
\draw(4,0)--(4,8.2);
\draw(5,0)--(5,8.2);
\draw(6,0)--(6,8.2);

\draw(0,1)--(6.2,1);
\draw(0,2)--(6.2,2);
\draw(0,3)--(6.2,3);
\draw(0,4)--(6.2,4);
\draw(0,5)--(6.2,5);
\draw(0,6)--(6.2,6);
\draw(0,7)--(6.2,7);
\draw(0,8)--(6.2,8);

\draw[red, thick, postaction={decorate},
      decoration={markings, mark=at position 0.5 with {\arrow{>}}}]
      (0,2) -- (1,0);

\draw[red, thick, postaction={decorate},
      decoration={markings, mark=at position 0.5 with {\arrow{>}}}]
      (0,3) -- (1,1);

\draw[red, thick, postaction={decorate},
      decoration={markings, mark=at position 0.5 with {\arrow{>}}}]
      (0,4) -- (1,2);

\draw[red, thick, postaction={decorate},
      decoration={markings, mark=at position 0.5 with {\arrow{>}}}]
      (0,5) -- (1,3);

\draw[blue, thick, postaction={decorate},
      decoration={markings, mark=at position 0.5 with {\arrow{>}}}]
      (0,0) -- (0,1);

\draw[blue, thick, postaction={decorate},
      decoration={markings, mark=at position 0.5 with {\arrow{>}}}]
      (0,1) -- (0,2);

\draw[blue, thick, postaction={decorate},
      decoration={markings, mark=at position 0.5 with {\arrow{>}}}]
      (0,2) -- (0,3);

\draw[blue, thick, postaction={decorate},
      decoration={markings, mark=at position 0.5 with {\arrow{>}}}]
      (0,3) -- (0,4);

\draw[blue, thick, postaction={decorate},
      decoration={markings, mark=at position 0.5 with {\arrow{>}}}]
      (0,4) -- (0,5);

\draw[blue, thick, postaction={decorate},
      decoration={markings, mark=at position 0.5 with {\arrow{>}}}]
      (0,5) -- (0,6);

\draw[blue, thick, postaction={decorate},
      decoration={markings, mark=at position 0.5 with {\arrow{>}}}]
      (0,6) -- (0,7);

\draw[blue, thick, postaction={decorate},
      decoration={markings, mark=at position 0.5 with {\arrow{>}}}]
      (0,7) -- (0,8);

\draw[blue, thick, postaction={decorate},
      decoration={markings, mark=at position 0.5 with {\arrow{>}}}]
      (1,1) -- (1,2);

\draw[blue, thick, postaction={decorate},
      decoration={markings, mark=at position 0.5 with {\arrow{>}}}]
      (1,0) -- (1,1);

\draw[blue, thick, postaction={decorate},
      decoration={markings, mark=at position 0.5 with {\arrow{>}}}]
      (1,2) -- (1,3);

\draw[blue, thick, postaction={decorate},
      decoration={markings, mark=at position 0.5 with {\arrow{>}}}]
      (1,3) -- (1,4);

\draw[blue, thick, postaction={decorate},
      decoration={markings, mark=at position 0.5 with {\arrow{>}}}]
      (1,4) -- (1,5);

\draw[blue, thick, postaction={decorate},
      decoration={markings, mark=at position 0.5 with {\arrow{>}}}]
      (2,0) -- (2,1);

\draw[blue, thick, postaction={decorate},
      decoration={markings, mark=at position 0.5 with {\arrow{>}}}]
      (2,1) -- (2,2);

\draw[blue, thick, postaction={decorate},
      decoration={markings, mark=at position 0.5 with {\arrow{>}}}]
      (2,2) -- (2,3);

\draw[blue, thick, postaction={decorate},
      decoration={markings, mark=at position 0.5 with {\arrow{>}}}]
      (2,3) -- (2,4);

\draw[blue, thick, postaction={decorate},
      decoration={markings, mark=at position 0.5 with {\arrow{>}}}]
      (3,0) -- (3,1);

 \draw[blue, thick, postaction={decorate},
      decoration={markings, mark=at position 0.5 with {\arrow{>}}}]
      (3,1) -- (3,2);

\draw[red, thick, postaction={decorate},
      decoration={markings, mark=at position 0.5 with {\arrow{>}}}]
      (1,2) -- (2,0);

\draw[red, thick, postaction={decorate},
      decoration={markings, mark=at position 0.5 with {\arrow{>}}}]
      (1,3) -- (2,1);

\draw[red, thick, postaction={decorate},
      decoration={markings, mark=at position 0.5 with {\arrow{>}}}]
      (1,4) -- (2,2);

\draw[red, thick, postaction={decorate},
      decoration={markings, mark=at position 0.5 with {\arrow{>}}}]
      (2,2) -- (3,0);

\draw[red, thick, postaction={decorate},
      decoration={markings, mark=at position 0.5 with {\arrow{>}}}]
      (2,3) -- (3,1);

\draw[red, thick, postaction={decorate},
      decoration={markings, mark=at position 0.5 with {\arrow{>}}}]
      (3,2) -- (4,0);

\draw[red, thick, postaction={decorate},
      decoration={markings, mark=at position 0.5 with {\arrow{>}}}]
      (2,4) -- (3,2);

\draw[red, thick, postaction={decorate},
      decoration={markings, mark=at position 0.5 with {\arrow{>}}}]
      (1,6) -- (2,4);

\draw[red, thick, postaction={decorate},
      decoration={markings, mark=at position 0.5 with {\arrow{>}}}]
      (0,8) -- (1,6);

\draw[red, thick, postaction={decorate},
      decoration={markings, mark=at position 0.5 with {\arrow{>}}}]
      (1,5) -- (2,3);

\draw[red, thick, postaction={decorate},
      decoration={markings, mark=at position 0.5 with {\arrow{>}}}]
      (0,7) -- (1,5);

\draw[red, thick, postaction={decorate},
      decoration={markings, mark=at position 0.5 with {\arrow{>}}}]
      (0,6) -- (1,4);


\node at (-.3,1.2) {1};
\node at (-.3,2.2) {1};
\node at (-.3,3.2) {1};
\node at (-.3,4.2) {1};
\node at (-.3,5.2) {1};
\node at (-.3,6.2) {1};
\node at (-.3,7.2) {1};
\node at (-.3,8.2) {1};

\node at (1,-.3) {1};
\node at (2,-.3) {3};
\node at (3,-.3) {12};
\node at (4,-.3) {55};
\node at (5,-.3) {273};

\node at (0,-.3) {1};
\node at (1,1.3) {1};
\node at (2,2.3) {1};

\node at (1,-.3) {1};
\node at (2,-.3) {3};
\node at (3,-.3) {12};
\node at (4,-.3) {55};
\node at (5,-.3) {273};

\node at (1,1.2) {2};

\node at (2,1.2) {7};
\node at (1,2.2) {3};

\node at (3,1.2) {30};
\node at (2,2.2) {12};
\node at (1,3.2) {4};

\node at (4,1.2) {143};
\node at (3,2.2) {55};
\node at (2,3.2) {18};
\node at (2,4.2) {25};
\node at (1,4.2) {5};
\end{tikzpicture}
\end{center}
\caption{The Fuss-Catalan matrix for $r=3$ $(t(x), t(x))$ as a path matrix}\label{TerM}
\end{figure}

\section{Further directions}
The fact that we have factorizations means that we can easily introduce extra parameters into our constructions. For instance, from the expression
$$\sum_{j=0}^n \frac{rj+1}{(r-1)n+j+1} \binom{rn}{n-j}\binom{k}{j}$$ for the Fuss-Catalan matrix we can proceed to introduce another parameter
$$\sum_{j=0}^n \frac{rj+1}{(r-1)n+j+1}s^{n-j} \binom{rn}{n-j}\binom{k}{j}.$$
For $r=2$, we obtain a generalized Fuss-Catalan matrix that begins
$$\left(
\begin{array}{cccc}
 1 & 1 & 1 & 1  \\
 s & s+1 & s+2 & s+3  \\
 2s^2 & 2s^2+3s & 2s^2+6s+1 & 2s^2+9s+3  \\
 5s^2 & 5s^3+9s^2 & 5s^3+18s^2+5s & 5s^3+27s^2+15s+1  \\
\end{array}\right).$$
This is the rectification of the Riordan array
$$\left(\frac{1-\sqrt{1-4sx}}{2sx}, \frac{1-2s(1-s)x-\sqrt{1-4sx}}{2s^2}\right)=(g_2(sx), xg_2(sx))\cdot (1,x(1+(1-s)x)).$$
Thus even in this simple case, we see that new features are appearing through this parameterizations.

Far-reaching generalizations of the Fuss-Catalan numbers have been explored, notably using the Fuss-Catalan-Qi numbers \cite{Qi1}. The techniques of this note may prove useful in their further study.

Another direction for generalization is to use other banded matrices in the ``right binomial transform'' process. To put this in a general context, we have the following result.
\begin{proposition} The down-shift of the product of the Riordan array $(g(x), f(x))$ with the transpose of the binomial matrix $\mathbf{B}^T$ is given by the Riordan array
$$(g(x), x(1+f(x)).$$
\end{proposition}
\begin{proof}
The generating function of $(g(x), f(x))$ is given by $\frac{g(x)}{1-yf(x)}$. The generating function of the product of this array by $\mathbf{B}^T$ on the right is given by
$$\frac{1}{1-y} \frac{g(x)}{1-\frac{y}{1-y}f(x)}=\frac{g(x)}{1-y-yf(x)}.$$ Down-shifting this ($y\to xy$) gives us the generating function
$$\frac{g(x)}{1-xy-xyf(x)}=\frac{g(x)}{1-yx(1+f(x))}.$$
This is the generating function of the Riordan array $(g(x), x(1+f(x))$.
\end{proof}
\begin{example}
We consider the banded production matrix
$$\left(
\begin{array}{ccccccc}
0 &1 & 0 & 0 & 0 & 0 & 0 \\
0 & 0 & 1 & 0 & 0 & 0 & 0 \\
1 & 0 & 0 & 1 & 0 & 0 & 0 \\
0 & 1 & 0 & 0 & 1 & 0 & 0 \\
0 & 0 & 1 & 0 & 0 & 1 & 0 \\
0 & 0 & 0 & 1 & 0 & 0 & 1 \\
0 & 0 & 0 & 0 & 1 & 0 & 0 \\
\end{array}\right).$$
We are thus considering the Riordan array $\left(\frac{1}{1+x^3}, \frac{x}{1+x^3}\right)^{-1}=(t(x^3), xt(x^3))$.
Multiplying by $\mathbf{B}^T$ on the right leads us to the square matrix $(t(x^3), 1+t(x^3))$, and to its downshift, $(t(x^3), x(1+t(x^3))$. We find that
$$(t(x^3), x(1+t(x^3))^{-1}=\left(\frac{1+3x+(1+2x)\sqrt{1+2x-3x^2}}{2(1+3x)}, \frac{2x^2}{1+3x-\sqrt{1+2x+3x^2}}\right).$$ 
This is turn is equal to the product 
$$\left(\frac{1}{1+x^3}, \frac{x}{1+x^3}\right)^{-1}\cdot \left(1, \frac{x}{1-x+x^2}\right).$$ In terms of sequence discovery, it is interesting to note that
$g(x)=\frac{1+3x+(1+2x)\sqrt{1+2x-3x^2}}{2(1+3x)}$ expands to give a sequence that begins
$$1, 0, 0, -1, 3, -9, 26, -75, 216, -623, 1800,\ldots.$$
The Hankel transform \cite{Layman} of this sequence is then given by
$$1, 0, -1, 2, 0, -2, 3, 0, -3, 4, 0, -4, 5, 0, -5, 6, 0,\ldots.$$
\end{example}
\begin{example}
We consider the banded production matrix
$$\left(
\begin{array}{ccccccc}
1 & 1 & 0 & 0 & 0 & 0 & 0 \\
1 & 1 & 1 & 0 & 0 & 0 & 0 \\
1 & 1 & 1 & 1 & 0 & 0 & 0 \\
0 & 1 & 1 & 1 & 1 & 0 & 0 \\
0 & 0 & 1 & 1 & 1 & 1 & 0 \\
0 & 0 & 0 & 1 & 1 & 1 & 1 \\
0 & 0 & 0 & 0 & 1 & 1 & 1 \\
\end{array}\right).$$ This generates the Bell Riordan array
$\left(\frac{f(x)}{x}, f(x)\right)$ where $f(x)=\Rev\{\frac{x}{1+x+x^2+x^3}\}$ (\seqnum{A036765}). We find that
$$\left(\frac{f(x)}{x}, x(1+f(x))\right)^{-1}=\left(\frac{1+x-2x^2+(1+x)\sqrt{1-4x^2}}{2(1+2x)}, x\frac{1+2x+\sqrt{1-4x^2}}{2(1+2x)}\right).$$ This is equal to the product 
$$\left(\frac{x}{1+x+x^2+x^3}, \frac{x}{1+x+x^2+x^3}\right)^{-1} \cdot \left(1,\frac{x}{1+x^2}\right).$$ 
Then $\frac{1+2x+\sqrt{1-4x^2}}{2(1+2x)}$ expands to give the sequence
$$1, -1, 1, -2, 3, -6, 10, -20, 35, -70, 126,\ldots$$ which is an alternating sign version of \seqnum{A210736}. Let $a_n$ denote the sequence $1,-1,1,-2,\ldots$. The Hankel transform of $a_n$ is the sequence $1,0,-1,0,1,0,-1,0,\ldots$. The Hankel transform of $a_{n+1}$ is $-(-1)^n$, and the Hankel transform of $a_{n+2}$ is $(-1)^{\binom{n+1}{2}}$. Finally, the Hankel transform of $0,1,-1,1,-2,\ldots$ is
$$0, -1, 1, -2, 2, -3, 3, -4, 4, -5, 5,\ldots.$$
The first component, $\frac{1+x-2x^2+(1+x)\sqrt{1-4x^2}}{2(1+2x)}$, expands to give a sequence that begins
$$1, -1, 0, -1, 1, -3, 4, -10, 15, -35, 56,\ldots,$$ where the sequence that begins
$$ 0,1, 1, 3, 4, 10, 15, 35, 56,\ldots$$ is $\binom{n}{\lfloor \frac{n-1}{2} \rfloor}$, \seqnum{A037952}. The Hankel transform of the sequence $1, -1, 0, -1, 1, -3, 4, -10,\ldots$ begins
$$1, -1, -2, 3, 0, -3, 4, -1, -5, 6, 0, -6, 7, -1, -8, 9,\ldots.$$ The generating function of this sequence is
$$\frac{1-x^2}{(1-x+x^2)(1+x+x^2)^2}.$$
\end{example}
\begin{example} As a final example, we take the case of the Riordan array
$$\left(\frac{1}{1+2x+2x^2+x^3}, \frac{x}{1+2x+2x^2+x^3}\right)^{-1}.$$ This will have a production matrix which begins
$$\left(
\begin{array}{ccccccc}
2 & 1 & 0 & 0 & 0 & 0 & 0 \\
2 & 2 & 1 & 0 & 0 & 0 & 0 \\
1 & 2 & 2 & 1 & 0 & 0 & 0 \\
0 & 1 & 2 & 2 & 1 & 0 & 0 \\
0 & 0 & 1 & 2 & 2 & 1 & 0 \\
0 & 0 & 0 & 1 & 2 & 2 & 1 \\
0 & 0 & 0 & 0 & 1 & 2 & 2 \\
\end{array}\right).$$  We find that the result of multiplying by the transpose of the binomial matrix on the right and downshifting results in the Riordan array given by the product
$$\left(\frac{1}{1+2x+2x^2+x^3}, \frac{x}{1+2x+2x^2+x^3}\right)^{-1} \cdot \left(1, \frac{x}{1+x+x^2}\right).$$ This new matrix has a production that begins
$$\left(
\begin{array}{ccccccc}
2 & 1 & 0 & 0 & 0 & 0 & 0 \\
2 & 1 & 1 & 0 & 0 & 0 & 0 \\
3 & 1 & 1 & 1 & 0 & 0 & 0 \\
6 & 2 & 1 & 1 & 1 & 0 & 0 \\
13 & 4 & 2 & 1& 1 & 1 & 0 \\
30 & 9 & 4 & 2 & 1 & 1 & 1 \\
72 & 21 & 9 & 4 & 2 & 1 & 1 \\
\end{array}\right).$$
The first column is essentially the sum of consecutive Motzkin numbers, and the other columns are defined by the Motzkin numbers, pre-pended by $1$.
This new Riordan array is
$$\left(\frac{1-x-2x^2+\sqrt{1-2x-3x^2}}{2(1+x)}, \frac{x(1+x+\sqrt{1-2x-3x^2})}{2(1+x)}\right)^{-1}=(g(x), f(x)),$$ with
$$g(x)=\frac{1}{x}\left(\frac{2}{3}\sqrt{\frac{3-2x}{x}}\sin\left(\frac{1}{3}\sin^{-1}\left(\frac{x(7x+18)\sqrt{\frac{3-2x}{x}}}{2(2x-3)^2}\right)\right)-\frac{2}{3}\right),$$ and
$$f(x)=\frac{2}{3}\sqrt{x(3-2x)} \sin\left(\frac{1}{3}\sin^{-1}\left(\frac{(7x+18)\sqrt{x(3-2x)}}{2(2x-3)^2}\right)\right)+\frac{x}{3}.$$
\end{example}
In general, if we start with the Riordan array $\left(\frac{1}{1+ax+bx^2+cx^3}, \frac{x}{1+ax+bx^2+cx^3}\right)^{-1}$, whose production matrix begins 
$$\left(
\begin{array}{ccccccc}
a & 1 & 0 & 0 & 0 & 0 & 0 \\
b & a & 1 & 0 & 0 & 0 & 0 \\
c & b & a & 1 & 0 & 0 & 0 \\
0 & c & b & a & 1 & 0 & 0 \\
0 & 0 & c & b & a & 1 & 0 \\
0 & 0 & 0 & c & b & a & 1 \\
0 & 0 & 0 & 0 & c & b & a \\
\end{array}\right),$$ then multiplying by the transpose of the binomial matrix on the right and downshifting leads to the matrix 
$$\left(\frac{1}{1+ax+bx^2+cx^3}, \frac{x}{1+ax+bx^2+cx^3}\right)^{-1}\cdot \left(1, \frac{x(1+x)}{1+ax+bx^2+cx^3}\right).$$
We have the following general result.
\begin{proposition} Given a Riordan array $(g(x), f(x))$, then the result of multiplying $(g(x), f(x))^{-1}$ by the transpose of the binomial matrix on the right and then downshifting is given by 
$$(g(x), f(x))^{-1} \cdot (1, (1+x)f).$$ 
\end{proposition}
\begin{proof} We have $(g(x), f(x))^{-1}=\left(\frac{1}{g(\bar{f})}, \bar{f}\right)$. Multiplying by the transpose of the binomial on the right gives us the matrix with generating function 
$$ \left(\frac{1}{1-y}, \frac{y}{1-y}\right)\cdot \frac{\frac{1}{g(\bar{f})}}{1-y \bar{f}}=\frac{\frac{1}{g(\bar{f})}}{1-y-y\bar{f}}.$$ 
Downshifting gives us the Riordan array with generating function 
$$\frac{\frac{1}{g(\bar{f})}}{1-xy-xy\bar{f}}=\frac{\frac{1}{g(\bar{f})}}{1-xy(1+\bar{f})}.$$ 
This is the Riordan array $\left(\frac{1}{g(\bar{f})}, x(1+\bar{f})\right)$. We then have 
\begin{align*} \left(\frac{1}{g(\bar{f})}, x(1+\bar{f})\right)&=\left(\frac{1}{g(\bar{f})}, f(\bar{f})(1+\bar{f})\right)\\
&=\left(\frac{1}{g(\bar{f})}, \bar{f}\right)\cdot (1, (1+x)f)\\
&=(g(x), f(x))^{-1} \cdot (1, (1+x)f).\end{align*}
\end{proof}
\begin{corollary} $$(g_r(x), x g_r(x))=\left(\frac{1}{1+x}, \frac{x}{(1+x)^r}\right)^{-1} \cdot \left(1, \frac{x}{(1+x)^{r-1}}\right).$$
\end{corollary}
\begin{proof} By the above, we have 
$$(g_r(x), x g_r(x))=\left(\frac{1}{1+x}, \frac{x}{(1+x)^r}\right)^{-1} \cdot \left(1, \frac{x(1+x)}{(1+x)^{r}}\right).$$
\end{proof}

\section{Conclusions}
The Fuss-Catalan numbers have attracted much attention in the literature (see \cite{He_CF, Qi1, Yang} and the references therein). In this note we have approached them through the optic of the Riordan group of matrices. We have identified an ancillary matrix, namely $(g_r(x), gr(x)^r)$ as being of importance in this study. This brings into play the notion of $d$-orthogonality as a related area, which may need more study in terms of its relationship to the Fuss-Catalan numbers, and, for instance, their relations to lattice paths.

\bigskip
\hrule

\noindent 2010 {\it Mathematics Subject Classification}:
Primary 05A15; Secondary 15B36, 11B37, 11B83, 11C20, 42C05, 11Y55
\noindent \emph{Keywords:} Catalan number, Fuss-Catalan numbers, Riordan array, generating function, $d$-orthogonal polynomials, integer sequences, lattice path.

\bigskip
\hrule
\bigskip
\noindent (Concerned with sequences
\seqnum{A000012},
\seqnum{A000108},
\seqnum{A001764},
\seqnum{A002293},
\seqnum{A036765},
\seqnum{A037952},
\seqnum{A039599},
\seqnum{A109956},
\seqnum{A109962}, and
\seqnum{A110616}).

\end{document}